\documentclass[11pt]{amsart}
\usepackage[headings]{fullpage}

\usepackage{amssymb,epic,eepic,epsfig,amsbsy,amsmath,amscd,color}
\usepackage[all]{xy}
\usepackage{graphicx}
\usepackage{texdraw}
\usepackage{url}
\usepackage{bbm}
\usepackage[colorlinks=true,linkcolor=blue,citecolor=blue]{hyperref}
\usepackage{mathrsfs}
\usepackage{tikz-cd}
\usepackage{tikz}
\usetikzlibrary{cd}
\usepackage{accents}
\usepackage[normalem]{ulem}
\usepackage{mathscinet}
\usepackage{aliascnt}
\usepackage{pinlabel}
\usepackage{amsthm}

\numberwithin{thmcounter}{section}
\newaliascnt{thmauto}{thmcounter}

\newaliascnt{Defauto}{thmcounter}

\newaliascnt{exauto}{thmcounter}

\newaliascnt{lemauto}{thmcounter}

\newaliascnt{propauto}{thmcounter}

\newaliascnt{corauto}{thmcounter}

\newaliascnt{remauto}{thmcounter}

\newaliascnt{clmauto}{thmcounter}

\newcounter{athmcounter}

\newaliascnt{athmauto}{athmcounter}

\newaliascnt{alemauto}{athmcounter}

\newaliascnt{acorauto}{athmcounter}

\newtheorem{atheorem}[athmauto]{Theorem}

\newtheorem{acor}[acorauto]{Corollary}

\newtheorem{theorem}[thmauto]{Theorem}
\newtheorem{lemma}[lemauto]{Lemma}
\newtheorem{proposition}[propauto]{Proposition}
\newtheorem{corollary}[corauto]{Corollary}
\newtheorem{claim}[clmauto]{Claim}

\theoremstyle{definition}

\newtheorem{remark}[remauto]{Remark}

\DeclareMathOperator{\Sp}{Sp}
\DeclareMathOperator{\SL}{SL}
\DeclareMathOperator{\GL}{GL}
\DeclareMathOperator{\lcm}{lcm}

\providecommand{\Z}{\ensuremath\mathbb Z}
\providecommand{\N}{\ensuremath\mathbb N}

\begin{document}

\title{Quotients of braid groups by their congruence subgroups}
\author{Wade Bloomquist}
\author{Peter Patzt}
\author{Nancy Scherich}

\maketitle

\begin{abstract}
The congruence subgroups of braid groups  arise from a congruence condition on the integral Burau representation $B_n \to \GL_{n}(\Z)$. We find the image of such congruence subgroups in $\GL_{n}(\Z)$---an open problem posed by Dan Margalit. 
Additionally, we characterize the quotients of braid groups by their congruence subgroups in terms of symplectic congruence subgroups. 
\end{abstract}

\section{Introduction}

The congruence subgroups of braid groups  arise from a congruence condition on the integral Burau representation. They are finite-index normal subgroups of the braid group
and give insight to the braid groups.  This generalizes both how congruence subgroups are used in the study of integral linear groups (like $\SL_n(\Z)$) and how pure braid groups are used in the study of braid groups.

The \emph{integral Burau representation} $\rho\colon  B_n\rightarrow \GL_{n}(\mathbb{Z})$ is the (unreduced) Burau representation specialized at $t=-1$,
\[\rho\colon  B_n\xrightarrow{\text{Burau}}{}\GL_{n}(\mathbb{Z}[t^{\pm1}])\xrightarrow{t=-1}{} \GL_{n}(\mathbb{Z}).\]
The \emph{level $\ell$ congruence subgroup} of the braid group on $n$ strands $B_n$, denoted $B_n[\ell]$, is defined to be the preimage of the level $\ell$ congruence subgroup of the general linear group, or more explicitly, the kernel of the following composition:
    \[B_n[\ell]:= \ker\bigg(B_n \stackrel{\rho}{\longrightarrow} \GL_{n}(\Z)\xrightarrow{\!\!\!\!\!\!\mod\ell}{} \GL_{n}(\Z/\ell\Z) \bigg) = \{ b\in B_n \mid \rho(b) \equiv I_{n} \mod \ell\}.\]

These subgroups have been studied by many authors, e.g. A'Campo \cite{ACampo}, Appel--Bloomquist--Gravel--Holden \cite{ABGH}, Arnol\cprime{}d \cite{A68}, Assion \cite{Assion},  Brendle \cite{brendle}, Brendle--Margalit \cite{BM18}, Kordek--Margalit \cite{KM19}, McReynolds \cite{McReynolds}, Nakamura \cite{Nakamura}, and Stylianakis \cite{S18}. 
Despite the extensive study, many questions about these groups remain open. 
For example, except for $B_n[2]$, which Arnol\cprime{}d  proved to be the pure braid group on $n$ strands \cite{A68} and the first rational homology of $B_n[4]$ found by Kordek--Margalit \cite{KM19}, their group homology (even their abelianization) is generally unknown.

Various open questions about the integral Burau representation can be found in Section 3 of Margalit's problem list \cite{margalitproblems} and Brendle's mini-course notes \cite{brendle}.
One fundamental theme of questioning is to understand the image of the integral Burau representation through relevant restrictions and quotients.\\

In this paper we provide answers to the following three \emph{questions}:

What is the image of 
\begin{enumerate}
    \item the integral Burau representation $\rho\colon  B_n\rightarrow \GL_{n}(\mathbb{Z})$?
    \item the integral Burau representation reduced modulo $\ell$, $B_n\rightarrow \GL_{n}(\mathbb{Z/\ell\Z})$?
    \item $B_n[\ell]$ under $\rho$, $B_n[\ell]\rightarrow \GL_{n}(\mathbb{Z})$? (Problem 3.4 in \cite{margalitproblems})\\
\end{enumerate}

Question (3), or Problem 3.4 in \cite{margalitproblems},  has been solved for $B_n[2]$ by Brendle--Margalit \cite{BM18} and the solution for even level is an easy corollary which has been pointed out in \cite{margalitproblems}. In this paper, we answer Question (3) completely and the result is stated below in \autoref{thm:prob3.4}. Question (1) is in fact a special case of Question (3) at level $\ell = 1$, which we answer in \autoref{cor:F}.

We also answer Question (2) and describe the image of $B_n \to \GL_{n}(\Z/\ell\Z)$, or equivalently the quotients $B_n/B_n[\ell]$ in \autoref{thm:mainresult}.
There has already been  considerable progress on this problem which we outline here. 
For all $n$, each odd $\ell$, and odd prime $p$,
\begin{itemize}
    \item $B_n/B_n[2]\cong S_n$, the symmetric group on $n$ letters \hfill  \cite{A68} 
    \item $B_{2n+1}/B_{2n+1}[p]\cong \Sp_{2n}(\mathbb{Z}/p\mathbb{Z})$, the symplectic group \hfill \cite{ACampo}
    \item  $B_n[\ell]/B_n[2\ell] \cong S_n$ \hfill \cite{ABGH}  and \cite{S18}
    \item  
$B_n[2\ell]/B_n[4\ell] \cong (\Z/2\Z)^{{n \choose 2 }}.$ \hfill\cite{ABGH}  and \cite{BM18}

\item 
$B_n[\ell]/B_n[4\ell] \cong B_n/B_n[4]$,  \hfill\cite{ABGH} and \cite{KM19}\\
 which is a non-split extension of $S_n$ by $(\Z/2\Z)^{{n \choose 2 }}.$
\end{itemize}

\subsection*{Answering Question (2)}

In our first main theorem, \autoref{thm:mainresult}, we unify these results and show how these quotients are related to the congruence subgroups of the symplectic groups. 
For this, we introduce some notation. Let 
\[ \Gamma_{2g}[\ell] := \ker\bigg( \Sp_{2g}(\Z) \xrightarrow{\!\!\!\!\!\!\mod\ell}\Sp_{2g}(\Z/\ell\Z)  \bigg)= \{ A \in \Sp_{2g}(\Z)  \mid A \equiv I_{2g} \mod \ell\}.\]
denote the \emph{level $\ell$ congruence subgroup} of $ \Sp_{2g}(\Z)$. Denote the subgroups of $ \Sp_{2g}(\Z)$ and $\Sp_{2g}(\Z/\ell\Z)$ that fix the first standard basis vector $e_1$ by  $[\Sp_{2g}(\Z)]_{e_1}$ and $[\Sp_{2g}(\Z/\ell\Z)]_{e_1}$, respectively. We further denote the \emph{level $\ell$ congruence subgroup} of $ [\Sp_{2g}(\Z)]_{e_1}$ by
\[ \Gamma_{2g-1}[\ell] := \ker\bigg( [\Sp_{2g}(\Z)]_{e_1} \xrightarrow{\!\!\!\!\!\!\mod\ell} [\Sp_{2g}(\Z/\ell\Z)]_{e_1} \bigg)= \{ A \in [\Sp_{2g}(\Z)]_{e_1}  \mid A \equiv I_{2g} \mod \ell\}.\]
To unify the notation, we define
\[ \Gamma_n := \begin{cases} \Sp_{2g}(\Z)&\text{for $n=2g$} \\
[\Sp_{2g}(\Z)]_{e_1} &\text{for $n=2g-1$.}\end{cases}\]
With this notation, we also get that
\[ \Gamma_n/\Gamma_n[\ell] \cong \begin{cases} \Sp_{2g}(\Z/\ell\Z)&\text{for $n=2g$} \\
[\Sp_{2g}(\Z/\ell\Z)]_{e_1} &\text{for $n=2g-1$.}\end{cases}\]

For those familiar with the notation in Brendle--Margalit,  $[\Sp_{2g+2}(\Z)]_{e_1}$ is isomorphic to $[\Sp_{2g+2}(\Z)]_{\vec{y}_{g+1}}$ in \cite{BM18} and we describe our basis labelling conventions in \autoref{sec:backgroundsubgroups}.

Finally, before stating our main results, we describe an inclusion of the symmetric group $S_n$ into $\Gamma_{n-1}[\ell]/\Gamma_{n-1}[2\ell] \cong \Gamma_{n-1}/\Gamma_{n-1}[2]$. From Arnol\cprime{}d \cite{A68}, we know that 
\[B_n/B_n[2] \cong S_n.\]
This is precisely the image of $B_n$ in 
\[ \Gamma_{n-1}/\Gamma_{n-1}[2] .\]

 Our first main theorem, providing an answer to Question 2, is as follows.
 
 \begin{atheorem}\label{thm:mainresult}For an integer $\ell=2^{k}m$ with $m$ odd,
   \[B_{n}/B_{n}[\ell]\cong \begin{cases}
       \Gamma_{n-1}/\Gamma_{n-1}[\ell] &\text{ for $n=2,3$,}\\
        B_n/B_n[2^k] \times\Gamma_{n-1}/\Gamma_{n-1}[m]  &\text{ for $n\geq 4$}.
        \end{cases}
\]
And for $n\ge 4$, $B_n/B_n[2^k]$ is trivial for $k=0$, isomorphic to $S_n$ for $k=1$, and for $k\ge2$ it is the non-split extension of $S_n$ by $\Gamma_{n-1}[2]/\Gamma_{n-1}[2^k]$ given by the preimage of $S_n\le \Gamma_{n-1}/\Gamma_{n-1}[2]$ in $\Gamma_{n-1}/\Gamma_{n-1}[2^k]$ via the quotient map.  
\end{atheorem}

\subsection*{Answering Questions (1) and (3)}

Our second main theorem solves Margalit's Problem 3.4 from \cite{margalitproblems}. After distributing a draft of this paper, Charalampos Stylianakis \cite{Stylperscomm} made us aware that he independently solved this problem which can now be found in \cite{Sty23}.

\begin{atheorem}\label{thm:prob3.4}
The image of $B_n[\ell]$ in $\GL_{n}(\Z)$ under the integral Burau representation is completely characterized as follows.
\begin{enumerate}
    \item If $\ell$ is even or $n=2,3$, the image is \[\Gamma_{n-1}[\ell].\]
    \item If $n\ge 4$ and $\ell$ is odd the image is the preimage of 
\[\begin{tikzcd}S_n \arrow[r, hook]& \Gamma_{n-1}[\ell]/\Gamma_{n-1}[2\ell] \end{tikzcd}\]
 along the quotient map $\Gamma_{n-1}[\ell]\longrightarrow \Gamma_{n-1}[\ell]/\Gamma_{n-1}[2\ell].$

\end{enumerate}
\end{atheorem}
Specializing the level to $\ell=1$ we arrive at an answer to Question 3, as follows.
\begin{acor}\label{cor:F}
   The image of $B_n[1]=B_n$ under the integral Burau representation 
   \begin{enumerate}
   \item for $n=2$ is $[\SL_2(\Z)]_{e_1} = [\Sp_2(\Z)]_{e_1} \cong \Z$,
       \item for $n=3$ is $\SL_2(\Z) = \Sp_2(\Z)$,
       \item for $n\ge 4$ is the preimage of $S_n$ along the quotient map $\Gamma_{n-1}\rightarrow \Gamma_{n-1}/\Gamma_{n-1}[2]$.
   \end{enumerate}
\end{acor}

\noindent \textbf{Acknowledgements.}
We would like to thank Tara Brendle, Dan Margalit, Jeremy Miller, Dan Petersen, Oscar Randal-Williams, Ismael Sierra, and Charalampos Stylianakis for helpful conversations. We would also thank the anonymous referee for helpful comments. This material is based upon work supported by the National Science Foundation under Grant No. DMS-1929284 while the third author was in residence at the Institute for Computational and Experimental Research in Mathematics in Providence, RI, during the Braids program. The first author is supported in part by NSF Grant DMS-1745583. The second author was supported by the Danish National Research Foundation through the Copenhagen Centre for Geometry and Topology (DNRF151) and a Simons Collaboration Grant.

\section{Background on the Burau representation and the symplectic groups}\label{sec:background}

We start with a  brief explanation of how to view the integral Burau representation as a symplectic representation. We explain some details of symplectic groups and their stabilizer subgroups. After that, we turn our attention to the congruence subgroups of the symplectic groups.

\subsection{Background on the integral Burau representation.}\label{sec:backgroundsubgroups}

In this section, we first give a brief introduction to viewing the integral Burau representation as the action of braid groups on the first homology of surfaces and how congruence subgroups fit into this framework. For a more detailed introduction, see Brendle \cite{brendle}.

As stated in the introduction, the \emph{integral Burau representation} $\rho\colon B_n\rightarrow \GL_{n}(\mathbb{Z})$ is the Burau representation specialized at $t=-1$,
\[\rho\colon B_n\xrightarrow{\text{Burau}}\GL_{n}(\mathbb{Z}[t^{\pm1}])\xrightarrow{t=-1} \GL_{n}(\mathbb{Z}).\]
More precisely, the Artin generators $\sigma_i$ (a half twist on the $i$-th and the $(i+1)$-st strands) are explicitly sent to the matrices
\[ \rho(\sigma_i)= \left( \begin{array}{c|cc|c} I_{i-1} & 0 & 0 & 0  \\ \hline 0 & 2 & -1 & 0  \\ 0 & 1 & 0 & 0   \\ \hline 0 & 0 & 0  & I_{n-i-1} \end{array} \right )\]
for $1\le i \le n-1$.

There is an alternating bilinear form on $\Z^n$ defined by 
\[ \langle e_i, e_j\rangle = \begin{cases} (-1)^{i+j+1}&i<j\\ 0 &i=j \\(-1)^{i+j} &i>j \end{cases}\]
on the standard basis $e_1,\dots, e_n$ of $\Z^n$.  Further, let $\Gamma_{n-1}$ be the subgroup of $\GL_n(\Z)$ defined as
\[ \Gamma_{n-1} = \{ A \in \GL_n(\Z) \mid \langle Av, Aw\rangle = \langle v,w\rangle \text{ for all $v,w\in \Z^n$, } Av_n = v_n\text{, and } w_n^T A = w_n^T\}\]
where 
\[ v_n = e_1 +\dots +e_n \quad\text{and}\quad w_n = e_1 - e_2 + \dots+ (-1)^{n-1}e_n.\]
It is immediate to check that $\rho(\sigma_i) \in \Gamma_{n-1}$ for all $1\le i \le n-1$ and we can therefore describe the integral Burau representation as a map
\[ \rho\colon B_n \longrightarrow \Gamma_{n-1}.\]
The following lemma explains why this group is isomorphic to the group with the same name from the introduction.

\begin{lemma}\label{lem:Gamma}
\[ \Gamma_n \cong \begin{cases} \Sp_{2g}(\Z)&\text{for $n=2g$} \\
[\Sp_{2g}(\Z)]_{e_1} &\text{for $n=2g-1$.}\end{cases}\]
\end{lemma}

\begin{proof}
As defined above, $\Gamma_n$ is the group of isometries of $\Z^{n+1}$ that fix $v_{n+1} = e_1+\dots + e_{n+1}$ and $w_n^T = e_1^T - e_2^T + \dots+ (-1)^{n}e_{n+1}^T$. We now define
\[ a_i =\sum_{k=1}^{2i} e_k \quad\text{and}\quad b_i=\begin{cases} e_{2i} + e_{2i+1} & 2i \neq n+1,\\ e_{n+1} & 2i = n+1.\end{cases}\]
Clearly, the set of vectors $\{a_i\}_i \cup \{b_i\}_i$ gives a basis of $\Z^{n+1}$. Easy computations show that these vectors also have the properties
\[ \langle a_i, b_j\rangle = \delta_{i,j}, \quad \langle a_i, a_j\rangle = \langle b_i, b_j\rangle =0\]
for all (appropriate) $i,j$.

Let us first consider the case $n =2g$ is even. Let
\[ W = \{ v \in \Z^{n+1} \mid w_{n+1}^Tv = 0\}\]
be the kernel of the dual vector $w_{n+1}^T$. It is easy to see that 
 $a_1,b_1, \dots, a_g,b_g$ gives a symplectic basis of $W \subset \Z^{n+1}$. In fact, $\Z^{n+1}$ splits into the direct sum $W\oplus V$, where $W$ is a maximal symplectic subspace and $V$ is the unique one-dimensional isotropic summand spanned by $v_{n+1}$. $\Gamma_n$ preserves $W$ and fixes $V$. This implies that $\Gamma_n$ is isomorphic to the symplectic group of $\Z^{2g}$.

If on the other hand $n=2g-1$, then $a_1,b_1,\dots, a_g,b_g$ is a symplectic basis of $\Z^{n+1}$. And therefore, $\Gamma_n$ is contained in the subgroup of $\Sp_{2g}(\Z)$ that fixes $v_{n+1}$. We want to argue that all elements in this subgroup fix $w_{n+1}^T$ automatically. For that note, that $a_1,b_1, \dots, a_g =v_{n+1}$ are all contained in
\[ W = \{ v \in \Z^{n+1} \mid w_{n+1}^Tv = 0\}.\]
Therefore $W$ splits into $S \oplus V$, where $S$ is a maximal symplectic summand of $W$ spanned by $a_1,b_1, \dots, a_{g-1},b_{g-1}$ and $V$ is the unique one-dimensional isotropic summand of $W$. Let $A\in \Sp_{2g}(\Z)$ that fixes $v_{n+1}$, then all $a_i,b_i$ with $i<g$ have to be send to a linear combination $u$ of $a_1, b_1, \dots, a_g = v_{n+1}$ because otherwise $\langle v_{n+1}, u\rangle$ wouldn't be zero. This implies that $A$ preserves $W$, which means that it sends $w^T_{n+1}$ to a multiple of itself. Further, if $b_g$ is sent to
\[A b_g=  \sum_{i=1}^g \lambda_1 a_i + \mu_i b_i,\]
the coefficient $\mu_g$ is determined by
\[ \mu_g = \langle v_{n+1}, Ab_g\rangle = \langle v_{n+1}, b_g\rangle = 1.\]
This implies that
\[ w^T_{n+1} Ab_g = \mu_g = 1\]
because all other basis elements are in $W$. And thus $w^T_{n+1}A = w^T_{n+1}$. Finally, we can do a change of coordinates. Since $v_n$ is unimodular, and $\Sp_{2g}(\Z)$ acts transitively on all unimodular vectors of $\Z^{2g}$, $\Gamma_n$ is isomorphic to $[\Sp_{2g}(\Z)]_{e_1}$.
\end{proof}

Let us define the congruence subgroups of $\Gamma_{n-1}$ by
\[ \Gamma_{n-1}[\ell] = \{ A\in \GL_n(\Z) \mid \langle Av, Aw\rangle = \langle v,w\rangle , \, Av_n = v_n, \, w^T_nA = w^T_n ,\, A \equiv I_n \mod\ell\} \unlhd \Gamma_{n-1}.\]
From the previous lemma, we immediately see that this definition agrees with the one given in the introduction.

\subsection{Background on the reduced integral Burau representation}

The reduced Burau representation is one dimension smaller than the Burau representation. As we have seen in \autoref{sec:backgroundsubgroups}, each matrix in the image of $\rho\colon B_n \to \GL_n(\Z)$ fixes the dual vector $w_n^T = e_1 -e_2+ \dots + (-1)^{n-1}e_n$. 
Therefore the kernel
\[ W = \{ v \in \Z^{n} \mid w_{n}^Tv = 0\}\]
 is a subrepresentation of the integral Burau representation and 
\[c_1 =  e_1 + e_2, c_2= -(e_2+e_3), \dots , c_{n-1} = (-1)^n(e_{n-1} +  e_n)\]
is a basis of $W$. Using this basis, we get the \emph{reduced integral Burau representation}
\[ \bar \rho\colon B_n \longrightarrow \GL_{n-1}(\Z)\]
with
\begin{gather*}
\bar\rho(\sigma_1) =   \left( \begin{array}{cc|c}1 & 1 & 0 \\ 0 & 1 & 0 \\ \hline 0 & 0 & I_{n-3} \end{array} \right),\quad \bar\rho(\sigma_{n-1}) =  \left( \begin{array}{c|cc} I_{n-3} & 0 & 0 \\ \hline 0 & 1 & 0 \\ 0 &-1 & 1 \end{array} \right),\quad
\bar\rho(\sigma_i)= \left( \begin{array}{c|ccc|c} I_{i-2} & 0 & 0 & 0 & 0 \\ \hline 0 & 1 & 0 & 0 & 0 \\ 0 & -1 & 1 & 1 & 0 \\ 0 & 0 & 0 & 1 & 0 \\ \hline 0 & 0 & 0 & 0 & I_{n-i-2} \end{array} \right ) ,
\end{gather*}
for $1<i<n-1$. For these formulas, we assume $n\ge 3$, otherwise the reduced integral Burau representation $W=0$.

Note that the alternating bilinear form on $\Z^n$ from above restricts to $W$ by
\[ \langle c_i,c_j\rangle = \begin{cases} -1 &j=i+1,\\ 1 & j=i-1,\\ 0 &\text{otherwise.}\end{cases}\]
Using this basis, we get an isometric isomorphism from $W$  to $\Z^{n-1}$ whose alternating bilinear form is given by
\[ \langle e_i,e_j\rangle' = \begin{cases} -1 &j=i+1,\\ 1 & j=i-1,\\ 0 &\text{otherwise.}\end{cases}\]
Let us define
\[ \Gamma'_{n-1} = \{ A\in \GL_{n-1}(\Z) \mid \langle v,w\rangle' = \langle Av,Aw\rangle' \text{ for all $v,w\in\Z^{n-1}$ and $A v'_n = v'_n$ if $n$ is even}\}\]
where if $n$ is even
\[ v'_n = e_1+e_3+\dots+e_{n-1}\]
is the image of $v_n \in W$ via the isometry to $\Z^n$. 
Clearly, the image of $\bar\rho$ lies in $\Gamma'_{n-1}$. Note that if $n$ is even, $v_n\in W$ but it spans the unique isotropic subspace of $W$, so any isometric automorphism of $W$ sends $v_n$ to $\pm v_n$. If $n$ is odd, $v_n$ spans an isotropic summand in $\Z^n$ that is complementary to the symplectic summand $W$.

The following lemma describes $\Gamma'_n$ in terms of symplectic groups.

\begin{lemma}\label{lem:Gamma'}
\[\Gamma'_{2g}=\Sp_{2g}(\Z) \cong \Gamma_{2g}\]
\[\Gamma'_{2g-1} \cong \Sp_{2g-2}(\Z) \ltimes \Z^{2g-2}\]
\end{lemma}

\begin{proof}
If $n=2g+1$ is odd, $W$ and thereby $\Z^{n-1}$ with $\langle -,-\rangle'$ is symplectic. Therefore $\Gamma'_{2g} = \Sp_{2g}(\Z)$.

If $n=2g$ is even, let us define 
\[ a'_i =\sum_{k=1}^{i} e_{2k-1} \text{ for $1 \le i\le g$}\quad\text{and}\quad b'_i=-e_{2i}\text{ for $1 \le i \le g-1$.}\]
Note that these vectors are the images of $a_1,b_1, \dots, a_g$ defined in the proof of \autoref{lem:Gamma} and form a basis of $\Z^{n-1}$, where $a'_g = v'_n$ spans the unique isotropic summand and $a_1,b_1, \dots, a_{g-1},b_{g-1}$ is a symplectic basis of a maximal symplectic summand of $\Z^{n-1}$.

We define a surjection together with a section
\[\begin{tikzcd}  \Gamma'_{2g-1} \arrow[r, two heads] & \Sp_{2g}(\Z) \arrow[l, bend right=20]. \end{tikzcd}\]
Let $A$ be a matrix in $\Gamma'_{2g-1}$. We know it fixes $v'_n$, therefore $A$ induces an isometry on $\Z^{n-1}/V$, where $V$ is the span of $a'_g=v'_n$. We denote the images of $a'_1,b'_1, \dots, a'_{g-1},b'_{g-1}\in \Z^{n-1}$ in this quotient by $\bar a'_1,\bar b'_1, \dots, \bar a'_{g-1},\bar b'_{g-1}$. Because $v'_n$ is isotropic in $\Z^{n-1}$, the alternating bilinear form descends to the quotient and those vectors give a symplectic basis of $\Z^{n-1}/V$. Let $\bar A\in \Sp_{2g-2}(\Z)$ be the isometry induced by $A$ with respect to that basis. This mapping clearly gives a group homomorphism. For the section, let $B\in \Sp_{2g-2}(\Z)$. We define a matrix $\tilde B\in \Gamma'_{2g-1}$ by having it fix $a'_g=v'_n$ and act on $ a'_1, b'_1, \dots,  a'_{g-1}, b'_{g-1}$ the same way $B$ acts on $\bar a'_1,\bar b'_1, \dots, \bar a'_{g-1},\bar b'_{g-1}$. It follows immediately that this gives a section. The kernel of the surjection is given by matrices $A\in \Gamma'_{2g-1}$ such that 
\[ Aa'_i = a'_i + \lambda_iv'_n \quad\text{and}\quad  Ab'_i = b'_i + \mu_iv'_n\]
for $\lambda_i,\mu_i \in \Z$ and $1\le i \le g-1$. This implies that the kernel is isomorphic to $\Z^{2g-2}$. It is easy to see that $\Sp_{2g-2}(Z)$ acts on $\Z^{2g-2}$ by matrix multiplication.
\end{proof}

We would also like to compare $\Gamma'_{n-1}$ to $\Gamma_{n-1}$ when $n$ is even. The following lemma states the difference.

\begin{lemma}\label{lem:SESGammaodd}
There is a short exact sequence
\[ 0 \longrightarrow \Z \longrightarrow \Gamma_{2g-1}  \longrightarrow \Gamma'_{2g-1} \longrightarrow 1.\]
\end{lemma}

\begin{proof}
Next, let us define a surjection together with a set-theoretic section 
\[\begin{tikzcd}  \Gamma_{2g-1} \arrow[r, two heads] & \Gamma'_{2g-1} \arrow[l, bend right=20]. \end{tikzcd}\]
Given a matrix $A\in \Gamma_{2g-1}$, we know that it leaves $W$ invariant and thus restricting $A$ to $W$ (and using the isometry to $\Z^{n-1}$) we get an element of $\Gamma'_{2g-1}$. This is clearly a group homomorphism. That it is surjective follows from the construction of a section. Let $B$ be a matrix in $\Gamma'_{2g-1}$, define the matrix $\tilde B \in \Gamma_{2g-1}$ that acts on the basis $a_1,b_1, \dots, a_g,b_g$ (from the proof of \autoref{lem:Gamma})  by acting on $a_1,b_1, \dots, a_g$ the same way that $B$ acts on $a'_1,b'_1, \dots, a'_g$ and by sending $b_g$ to
\[ b_g + \sum_{i=1}^{g-1}\bigg( \langle b_g, \tilde Ba_i\rangle b_i - \langle b_g, \tilde Bb_i\rangle a_i \bigg).\] 
A straightforward calculation shows that $\tilde B$ is indeed an element of $\Gamma_{2g-1}$ and it is immediately clear that this construction gives a set-theoretic section. (One can compute that it is not a group homomorphism.) The kernel of the surjection is given by all $A\in \Gamma_{2g-1}$ that leave $a_1,b_1, \dots, a_g$ fixed. That means they are determined by where they send $b_g$. Since $Aa_1 = a_1, Ab_1= b_1, \dots, Aa_g = a_g, Ab_g$ has to remain a symplectic basis of $\Z^{2g}$, we conclude that $Ab_g = b_g + \lambda a_g$ for any $\lambda \in \Z$. Note that this kernel is isomorphic to the infinite cyclic group $\Z$. \end{proof}

\begin{remark}\label{rmk:composition}
Let $\psi\colon \Gamma_{n-1} \to \Gamma'_{n-1}$ be the isomorphism from \autoref{lem:Gamma'} for odd $n$ and the surjection from \autoref{lem:SESGammaodd} for even $n$. Then
\[ \bar \rho = \psi \circ \rho.\]
\end{remark}

\subsection{Multiplicativity of congruence subgroups of symplectic groups}\label{sec:TechnicalSymplectic}

The results about symplectic groups that we need are well understood and documented. We will summarize them in the following proposition.

\begin{proposition}\label{prop:symplecticcong}\ 
\begin{enumerate}
\item $\Sp_{2g}(\Z)/\Gamma_{2g}[\ell] \cong \Sp_{2g}(\Z/\ell\Z)$ \cite[Theorem VII.20]{newman}
\item $\Gamma_{2g}[\ell]\cap \Gamma_{2g}[m] = \Gamma_{2g}[\lcm(\ell,m)]$ and $\Gamma_{2g}[\ell]\cdot \Gamma_{2g}[m] = \Gamma_{2g}[\gcd(\ell,m)]$ \cite[Theorem VII.22]{newman}
\item $\Gamma_{2g}[\gcd(\ell,m)]/ \Gamma_{2g}[\ell] \cong \Gamma_{2g}[m]/\Gamma_{2g}[\lcm(\ell,m)]$ \cite[Theorem VII.23]{newman}
\end{enumerate}
\end{proposition}

We will now prove analogous results for the stabilizer subgroups $\Gamma_{2g+2} \cong [\Sp_{2g+2}(\Z)]_{e_1}$ and their congruence subgroups $\Gamma_{2g+2}[\ell]\cong \ker\big( [\Sp_{2g+2}(\mathbb{Z})]_{e_1} \to  [\Sp_{2g+2}(\mathbb{Z}/\ell\Z)]_{e_1}\big)$. These results are certainly not surprising to the experts, but we were unable to find them in the literature.

\begin{proposition}\label{prop:oddsymplecticcong}For $n$ odd,
\begin{enumerate}
\item $\Gamma_n/\Gamma_{n}[\ell] \cong [\Sp_{n+1}(\Z/\ell\Z)]_{e_1}$,
\item $\Gamma_{n}[\ell]\cap \Gamma_{n}[m] = \Gamma_{n}[\lcm(\ell,m)]$ and $\Gamma_{n}[\ell]\cdot \Gamma_{n}[m] = \Gamma_{n}[\gcd(\ell,m)]$, and
\item $\Gamma_{n}[\gcd(\ell,m)]/ \Gamma_{n}[\ell] \cong \Gamma_{n}[m]/\Gamma_{n}[\lcm(\ell,m)]$.
\end{enumerate}
\end{proposition}

\begin{proof}
Let us start with (1).

We need to show that $[\Sp_{n+1}(\mathbb{Z})]_{e_1} \to  [\Sp_{n+1}(\mathbb{Z}/\ell\Z)]_{e_1}$ is surjective. Let $A\in [\Sp_{n+1}(\Z/\ell\Z)]_{e_1}$ and let $B$ be the lower right $(n-2)\times (n-2)$ submatrix of $A$ with the following entries.
\[ A=\begin{bmatrix}
1 & x & a_3& \cdots & a_{n+1}\\
0 & 1 & 0 & \cdots & 0\\
\vline & \vline &\vline  & &\vline  \\
0 & v_2 & v_3 &\cdots &v_{n+1}  \\
\vline & \vline & \vline && \vline  \\
\end{bmatrix}
 \qquad B=\begin{bmatrix}
 \vline  & &\vline  \\
  v_3 &\cdots &v_{n+1}  \\
 \vline && \vline  \\
\end{bmatrix}\]
Clearly, $B \in \Sp_{n-2}(\Z/\ell\Z)$. Since $\Sp_{n-2}(\Z) \to \Sp_{n-2}(\Z/\ell\Z)$, we can choose a $\tilde B\in \Sp_{n-2}(\Z)$ so that $\tilde B \equiv B \mod \ell$. Also choose $\tilde{x}\in\Z$ so that $\tilde{x}\equiv x\mod \ell$, and  $\tilde{v}_2\in\Z^{n-1}$ so that $\tilde{v}_2\equiv v_2\mod \ell$. If $\tilde v_3, \dots, \tilde v_{n+1}$ are the columns of $\tilde B$, set $\tilde a_i = \langle \tilde v_2, \tilde v_i\rangle$ and set
\[\tilde{A}=\begin{bmatrix}
1 & \tilde{x} & \tilde{a}_3& \cdots & \tilde{a}_{n+1}\\
0 & 1 & 0 & \cdots & 0\\
\vline & \vline & & & \\
0 & \tilde{v}_2 &  &\tilde{B} &  \\
\vline & \vline &  &&   \\
\end{bmatrix}.\] 
This defines a matrix $\tilde A \in [\Sp_{n+1}(\Z)]_{e_1}$.
To check that $\tilde A \equiv A \mod \ell$, it only remains to check that $\tilde a_i \equiv a \mod \ell$. Note that
\[ \tilde a_i = \langle \tilde v_2, \tilde v_i\rangle \equiv  \langle  v_2, v_i\rangle = a_i \mod \ell.\]

We continue with (2).

Observe that $h\in \Gamma_n[\lcm(\ell,m)]$ if and only if $h=I+X$ such that all entries of $X$ are divisible by $\lcm(\ell,m)$ if and only if $h=I+X$ such that all entries of $X$ are divisible by $\ell$ and $m$ if and only if $h\in \Gamma_n[\ell]\cap \Gamma_n[m]$.

The second part of (2) is slightly more complicated. To prove this, we first prove the following claim.

\begin{claim}\label{Lem:y=mModl}
Let $n$ be odd and $d=\gcd(m,\ell)$. Then for every $A\in\Gamma_n[d]$, there exists $\tilde A\in \Gamma_n[m]$ so that $\tilde A\equiv A\mod\ell$.
\end{claim}

\begin{proof}
\renewcommand{\qedsymbol}{\ensuremath\blacksquare}
Let $A\in \Gamma_n[d]$ and let $B$ be the lower right $(n-2)\times (n-2)$ submatrix with the following entries.
\[ A=\begin{bmatrix}
1 & x & a_3& \cdots & a_{n+1}\\
0 & 1 & 0 & \cdots & 0\\
\vline & \vline &\vline  & &\vline  \\
0 & v_2 & v_3 &\cdots &v_{n+1}  \\
\vline & \vline & \vline && \vline  \\
\end{bmatrix}
 \qquad B=\begin{bmatrix}
 \vline  & &\vline  \\
  v_3 &\cdots &v_{n+1}  \\
 \vline && \vline  \\
\end{bmatrix}\]
Clearly, $B \in \Gamma_{n-2}[d]$. By Lemma 4 from Newman--Smart \cite{NS}, there exists $\tilde B\in \Gamma_{n-2}[m]$ with $\tilde B\equiv B\mod\ell$.   There exists $\tilde x \in\Z$ and $\tilde v_2 \in\Z^{n-1}$ so that $\tilde x\equiv x\mod\ell$, $\tilde x\equiv 0\mod m$, $\tilde v_2\equiv v_2\mod \ell$, and $\tilde v_2\equiv 0\mod m$. If $\tilde v_3, \dots, \tilde v_{n+1}$ are the columns of $\tilde B$, set $\tilde a_i = \langle \tilde v_2, \tilde v_i\rangle$ and set
\[\tilde{A}=\begin{bmatrix}
1 & \tilde{x} & \tilde{a}_3& \cdots & \tilde{a}_{n+1}\\
0 & 1 & 0 & \cdots & 0\\
\vline & \vline & & & \\
0 & \tilde{v}_2 &  &\tilde{B} &  \\
\vline & \vline &  &&   \\
\end{bmatrix}.\] 
This defines a matrix $\tilde A \in [\Sp_{n+1}(\Z)]_{e_1}$.
To check that $\tilde A\in \Gamma_n[m]$, it only remains to check that $\tilde a_i \equiv 0 \mod m$. Note that
\[ \tilde a_i = \langle \tilde v_2, \tilde v_i\rangle \equiv 0 \mod m.\]
To check that $\tilde A \equiv A \mod \ell$, it only remains to check that $\tilde a_i \equiv a_i \mod \ell$. Note that
\[ \tilde a_i = \langle \tilde v_2, \tilde v_i\rangle \equiv  \langle  v_2, v_i\rangle = a_i \mod \ell.\qedhere\]
\end{proof}

Now to prove the second part of (2), let $A \in \Gamma_n[\gcd(\ell,m)]$, and let $\tilde A \in \Gamma_n[m]$ be a matrix such that $\tilde A \equiv A \mod \ell$ by \autoref{Lem:y=mModl}. Therefore $A \cdot {\tilde A}^{-1} \equiv I \mod \ell$ or in other words $A \cdot {\tilde A}^{-1} \in \Gamma_n[\ell]$. So $A = (A \cdot {\tilde A}^{-1})\cdot \tilde A \in \Gamma_n[\ell] \cdot \Gamma_n[m]$.

Finally, Part (3) follows from (2) using a general isomorphism theorem:
\[ \Gamma_{n}[\gcd(\ell,m)]/ \Gamma_{n}[\ell] = (\Gamma_{n}[\ell)]\cdot  \Gamma_{n}[m])/ \Gamma_{n}[\ell]\cong \Gamma_{n}[m]/(\Gamma_{n}[\ell]\cap \Gamma_{n}[m] )= \Gamma_{n}[m]/\Gamma_{n}[\lcm(\ell,m)]\qedhere\]
\end{proof}

The following corollary is an immediate consequence of \autoref{prop:symplecticcong}, \autoref{prop:oddsymplecticcong}, and a general isomorphism theorem.

\begin{corollary} \label{cor:Gammaprod}
Let $n,m,\ell \in\N$ with $\gcd(m,\ell) = 1$. Then
 \[ \Gamma_n/\Gamma_n[m\ell] \cong \Gamma_n/\Gamma_n[m] \times \Gamma_n/\Gamma_n[\ell].\]
 \end{corollary}

\section{Congruence subgroups of braid groups and the main theorems}

This section focuses on the congruence subgroups of braid groups and the proofs of the main theorems. We start by proving \autoref{prop:braidcong} which is the analog of \autoref{prop:symplecticcong} the braid groups, and then provide proofs of \hyperref[thm:mainresult]{Theorems \ref{thm:mainresult}} and \ref{thm:prob3.4}. Much of this section is straightforward computations and rehashing of known work. 

\subsection{Multiplicativity of congruence subgroups of braid groups}

We start by proving the analog of \autoref{prop:symplecticcong} (2) for braid groups. This is straightforward but we couldn't find it in the existing literature.

\begin{proposition}\label{prop:braidcong}
$B_{n}[\ell]\cap B_{n}[m] = B_{n}[\lcm(\ell,m)]$ and $B_{n}[\ell]\cdot B_{n}[m] = B_{n}[\gcd(\ell,m)]$.
\end{proposition}

\begin{proof}
The first part is the same proof as in \autoref{prop:symplecticcong} (2):

$h\in B_n[\lcm(\ell,m)]$ if and only if $\rho(h)=I+X$ such that all entries of $X$ are divisible by $\lcm(\ell,m)$ if and only if $\rho(h)=I+X$ such that all entries of $X$ are divisible by $m$ and $\ell$ if and only if $h\in B_n[\ell]\cap B_n[m]$.

For the second part, we first assume that $\gcd(\ell,m)=1$. 
Because
\[ \begin{pmatrix} 2&-1\\1&0\end{pmatrix}^m = \begin{pmatrix} m+1&-m\\m&-m+1\end{pmatrix},\]
we get that $\rho(\sigma_i^m) $ is the identity modulo $m$ and thus $\sigma_i^m \in B_n[m]$.
By Bezout's identity there exists $d,r\in \mathbb{Z}$ so that $md+\ell r=1$.
Then $\sigma_i=\sigma_i^{md+\ell r}=(\sigma_i^m)^d\cdot (\sigma_i^\ell)^r\in  B_n[m]\cdot B_n[\ell]$, showing that $B_n\subseteq B_n[m]\cdot B_n[\ell]$.

For the general case, let 
\[ m = \prod_{p\text{ prime}} p^{a_p}\quad\text{and}\quad\ell = \prod_{p\text{ prime}} p^{b_p}. \quad\text{Then}\quad \gcd(m,\ell) =\prod_{p\text{ prime}} p^{\min(a_p,b_p)}. \]
We now define
\[ d_m = \prod_{\substack{p\text{ prime}\\a_p\le b_p}} p^{a_p}\quad\text{and}\quad d_\ell = \prod_{\substack{p\text{ prime}\\a_p>b_p}} p^{b_p}.\]
Let $h\in B_n[\gcd(m,\ell)]$. Because 
\[\gcd\left(\frac{m}{d_m},\frac{\ell}{d_\ell}\right) = 1,\]
we can write $h=h_1 \cdot h_2$ with $h_1 \in B_n[\frac{m}{d_m}]$ and $h_2 \in B_n[\frac{\ell}{d_\ell}]$. As both $h,h_1 \in B_n[d_\ell]$, so is $h_2$. Therefore, 
\[h_2 \in B_n\left[\frac{\ell}{d_\ell}\right] \cap B_n[d_\ell] = B_n\left[\lcm\left(\frac{\ell}{d_\ell},d_\ell\right)\right] = B_n[\ell].\]
Likewise, $h_1\in B_n[m]$.
\end{proof}

As for the symplectic groups, this immediately implies the following multiplicativity result.

\begin{corollary}\label{cor:multbraid} Let $n,m,\ell \in\N$ with $\gcd(m,\ell) = 1$. Then
 \[ B_n/B_n[m\ell] \cong B_n/B_n[m] \times B_n/B_n[\ell].\]
 \end{corollary}

\subsection{Surjections $B_n[\ell] \to \Gamma_{n-1}[\ell]$}

In this section, we shall see that $B_n[\ell] \to \Gamma_{n-1}[\ell]$ is a surjection if $\ell$ is even. This follows from work of A'Campo \cite{ACampo} and Brendle--Margalit \cite{BM18} for $n\ge 5$, and can easily be extended to all $n$, which we do here.

\begin{proposition}[Brendle--Margalit \cite{BM18}, Theorem 3.3] \label{lem:BM}
For $n\geq 3$, the restriction of $\rho$ to the pure braid group  $\rho\colon  PB_n=B_n[2]\rightarrow \Gamma_{n-1}[2]$ is surjective.
\end{proposition}

\begin{proof}
Brendle--Margalit \cite{BM18} prove this result for $n\geq 5$, however their proof can easily seen to work also for $n=3,4$. For completeness we will include the details here.

In Proposition 3.2 \cite{BM18}, they find generating sets for $\Gamma_{n-1}[2]$ for $n\ge 5$. In the proof of their Theorem 3.3, they then find preimages of each generator in $B_n[2]$. In the following paragraph, we will observe that the proof of their Proposition 3.2 also works for $n=3,4$. The preimages they find in the proof of their Theorem 3.3 then also work as preimages of the given generators.

Let us first look at the case $n=3$. By Mumford \cite[Proposition A.3, p.207]{Mumford}, $\Gamma_2[2]$ is generated by (in the notation of Brendle--Margalit)
\[ \tau^2_{\vec x_1}, \tau^2_{\vec y_1}, \tau^2_{\vec x_1+\vec y_1}.\]
Brendle--Margalit then replace $\tau^2_{\vec x_1+\vec y_1}$ with
\[ \tau^2_{\vec x_1-\vec y_1} = \tau^2_{\vec x_1}\tau^2_{\vec x_1+\vec y_1}\tau^{-2}_{\vec x_1}.\]
We see that all generators from their Proposition 3.2 (that make sense for $n=3$) still form a generating set of $\Gamma_2[2]$. 

They proceed by reducing $\Gamma_{2g+1}[2]$ to $\Gamma_{2g}[2]$. This works exactly the same way by reducing $\Gamma_3[2]$ to $\Gamma_2[2]$ and we get the generators
\[ \tau^2_{\vec x_1}, \tau^2_{\vec y_1}, \tau^2_{\vec x_1-\vec y_1}, \tau^2_{\vec y_2-\vec x_1}, \tau^2_{\vec y_2-\vec y_1}, \tau^2_{\vec y_2} \]
for $\Gamma_3[2]$. Again, these are exactly the generators described in their Proposition 3.2.
\end{proof}

The following direct corollary is well known to experts, and is stated by Margalit  \cite{margalitproblems} as a corollary of the work of A'Campo \cite{ACampo} in the case that $n$ is odd.

\begin{corollary}\label{cor:imageelleven}
For $n\geq 3$, the restriction of $\rho$ to $\rho\colon  B_n[\ell]\rightarrow \Gamma_{n-1}[\ell]$ is surjective if $\ell$ is even.
\end{corollary}

\begin{proof}
Let $x\in \Gamma_{n-1}[\ell] \le \Gamma_{n-1}[2]$. By  \autoref{lem:BM} (Theorem 3.3 of Brendle--Margalit \cite{BM18}), there is a $y \in B_n[2]$ that maps to $x$. Because $x \equiv I \mod \ell$, $y$ has to even be in $B_n[\ell]$.
\end{proof}

The case $n=3$ of \autoref{lem:BM} can be proved by direct computation, which we single out in the following lemma for later use.

\begin{lemma}\label{lem:n=3}
$\rho\colon B_2 \to \Gamma_1$ and $\rho\colon  B_3 \to \Gamma_2$ are surjective. 
\end{lemma}

\begin{proof}
For $\rho\colon B_2 \to \Gamma_1$, we observe that this map is actually an isomorphism and both groups are isomorphic to $\Z$.

For $n=3$, we show that 
\[\bar{\rho}\colon B_3 \to \Gamma'_2 = \Sp_2(\Z) = \SL_2(\Z)\cong \Gamma_2\]
 is surjective and conclude surjectivity of $\rho$ from \autoref{rmk:composition}. 
  The two matrices
\[ \bar{\rho}(\sigma_1)=\begin{pmatrix} 1&1\\0&1\end{pmatrix}\quad\text{and}\quad\bar{\rho}(\sigma_2)=\begin{pmatrix} 1&0\\-1&1\end{pmatrix}\]
 generate $\SL_2(\Z)$.
\end{proof}

\subsection{Quotients of braid groups by their congruence subgroups}

In this section, prove \autoref{thm:mainresult}.

\begin{proof}[Proof of \autoref{thm:mainresult}]
The cases $n=2,3$ follow immediately from \autoref{lem:n=3}. We now assume $n\ge 4$.

First, we note that
\[ B_n[\ell]/B_n[\ell m] \longrightarrow \Gamma_{n-1}[\ell]/\Gamma_{n-1}[\ell m]\]
is always injective. This can easily be seen as $B_n[\ell]$ maps to $\Gamma_{n-1}[\ell]$ through $\rho$ and $B_n[\ell m]$ is the kernel of $B_n[\ell] \to \Gamma_{n-1}[\ell]/\Gamma_{n-1}[\ell m]$.

We next observe that
\[ B_n[\ell]/B_n[\ell m] \longrightarrow \Gamma_{n-1}[\ell]/\Gamma_{n-1}[\ell m]\]
is an isomorphism if $\ell$ is even, because $B_n[\ell] \to \Gamma_{n-1}[\ell]$ is surjective by \autoref{cor:imageelleven}.

Furthermore, we want to show that
\[ B_n[\ell]/B_n[\ell m] \longrightarrow \Gamma_{n-1}[\ell]/\Gamma_{n-1}[\ell m]\]
is an isomorphism if $\ell m$ is odd. For this consider the following commutative diagram.
\[\begin{tikzcd}
B_n[2\ell]/B_n[2\ell m]\arrow[d] \arrow[r]&B_n[\ell]/B_n[\ell m]\arrow[d]\\
\Gamma_{n-1}[2\ell]/\Gamma_{n-1}[2\ell m]\arrow[r] & \Gamma_{n-1}[\ell]/\Gamma_{n-1}[\ell m] 
\end{tikzcd}
\]
In this diagram, the left downwards map is an isomorphism as we have just seen. Both across maps are also isomorphisms, which follows from \autoref{prop:symplecticcong} (2), \autoref{prop:oddsymplecticcong} (2), \autoref{prop:braidcong}, and a general isomorphism theorem. This proves that the right downwards map is also an isomorphism.

In the final case, $\ell$ is odd and $m$ is even. Then $B_n[\ell]/B_n[\ell m]$ of course is an extension of $B_n[2\ell]/B_n[\ell m] \cong \Gamma_{n-1}[2\ell]/\Gamma_{n-1}[\ell m]$ and $B_n[\ell]/B_n[2\ell]\cong B_n/B_n[\ell] \cong S_n$.

Putting all of these cases together and using \autoref{cor:Gammaprod} and \autoref{cor:multbraid}, proves 
   \[B_{n}/B_{n}[\ell]\cong \begin{cases}
       \Gamma_{n-1}/\Gamma_{n-1}[\ell] &\text{ for $n=2,3$,}\\
        B_n/B_n[2^k] \times\Gamma_{n-1}/\Gamma_{n-1}[m]  &\text{ for $n\geq 4$},
        \end{cases}
\]
for $\ell=2^{k}m$ with $m$ odd.

%
%
%
%
%

Let us now prove the statement about $B_n/B_n[2^k]$.
Consider the following inclusion of short exact sequences induced by $\rho$.
\begin{equation*}
\begin{tikzcd}
 1 \arrow[r]&  B_n[2]/B_n[2^k]\arrow[d,hook]\arrow[r] &  B_n/B_n[2^k]\arrow[d,hook] \arrow[r]  &B_n/B_n[2]\arrow[d,hook]\arrow[r]&  1 \\
1 \arrow[r]& \Gamma_{n-1}[2]/\Gamma_{n-1}[2^k]\arrow[r] &  \Gamma_{n-1}/\Gamma_{n-1}[2^k] \arrow[r]  &\Gamma_{n-1}/\Gamma_{n-1}[2]\arrow[r]&  1 
\end{tikzcd}\label{SES}
\end{equation*}
To prove the second part of the statement, we only need to find the image of 
\[\begin{tikzcd}B_n/B_n[2^k] \arrow[r,hook] &\Gamma_{n-1}/\Gamma_{n-1}[2^k].\end{tikzcd}\]
Clearly, the image is a subset of the preimage of the image of
\[\begin{tikzcd}B_n/B_n[2] \arrow[r,hook] &\Gamma_{n-1}/\Gamma_{n-1}[2].\end{tikzcd}\]
This gives us the morphism of short exact sequences
\begin{equation*}
\begin{tikzcd}
 1 \arrow[r]&  B_n[2]/B_n[2^k]\arrow[d,"\cong"]\arrow[r] &  B_n/B_n[2^k]\arrow[d] \arrow[r]  &B_n/B_n[2]\arrow[d,"\cong"]\arrow[r]&  1 \\
1 \arrow[r]& \Gamma_{n-1}[2]/\Gamma_{n-1}[2^k]\arrow[r] &  X\arrow[r]  &S_n\arrow[r]&  1 
\end{tikzcd}
\end{equation*}
where $X$ denotes the preimage of $S_n$ in $\Gamma_{n-1}/\Gamma_{n-1}[2^k]$. The proof of \autoref{thm:mainresult} explains that $B_n[2]/B_n[2^k]\to \Gamma_{n-1}[2]/\Gamma_{n-1}[2^k]$ is an isomorphism and Arnol\cprime{}d \cite{A68} proves that $B_n/B_n[2] \to S_n$ is an isomorphism. The five lemma then implies that $B_n/B_n[p^k] \to X$ is an isomorphism. 

This also implies the first statement, with the exception of noting that this extension does not split.  We conclude this from the following commutative diagram and the fact that $B_n/B_n[4]$ is a non-split extension as proven in Proposition 8.6 of \cite{KM19}. (They call denote $B_n/B_n[4]$ by $\mathcal Z_n$.)
\[\begin{tikzcd}
\Gamma_{n-1}[2]/\Gamma_{n-1}[2^k]\arrow[r,hook]\arrow[dd,two heads] & B_n/B_n[2^k]\arrow[dr,two heads]\arrow[dd,two heads] & \\
 & & S_n \\
 \Gamma_{n-1}[2]/\Gamma_{n-1}[4]\arrow[r,hook] & Z_n\arrow[ur,two heads] &
\end{tikzcd}\]
The commutative diagram comes from the quotient map
\[ \begin{tikzcd} B_n/B_n[2^k] \arrow[r,two heads] &B_n/B_n[4].\end{tikzcd}\]
If $B_n/B_n[2^k] \to S_n$ had a splitting $s \colon  S_n \to B_n/B_n[2^k]$, the composition
\[ S_n \stackrel{s}{\longrightarrow} B_n/B_n[2^k] \longrightarrow B_n/B_n[4]\]
would be a splitting of $B_n/B_n[4] \to S_n$, which contradicts Proposition 8.6 of \cite{KM19}.
\end{proof}

\subsection{Image of $B_n[\ell]$ in $\GL_n(\Z)$}

In this section, we prove \autoref{thm:prob3.4}.

\begin{proof}[Proof of \autoref{thm:prob3.4}]
For $n=3$,  \autoref{lem:n=3} proves that the image of $B_3[\ell] \to \Gamma_2$ is simply $\Gamma_2[\ell]$.

For $\ell$ even, \autoref{cor:imageelleven} implies that the image of $B_n[\ell] \to \Gamma_{n-1}$ is $\Gamma_{n-1}[\ell]$.

Let $n\ge 4$ and $\ell$ odd. Consider following the map of short exact sequences.
\begin{equation*}
\begin{tikzcd}
 1 \arrow[r]&  B_n[2\ell]\arrow[d,two heads]\arrow[r] &  B_n[\ell]\arrow[d] \arrow[r]  &S_n\arrow[d,hook]\arrow[r]&  1 \\
1 \arrow[r]& \Gamma_{n-1}[2\ell]\arrow[r] &  \Gamma_{n-1}[\ell]\arrow[r]  &\Gamma_{n-1}[\ell]/\Gamma_{n-1}[2\ell]\arrow[r]&  1 
\end{tikzcd}
\end{equation*}
We want to prove that the image of $B_n[\ell] \to \Gamma_{n-1}[\ell]$ is the preimage of $S_n \subset \Gamma_{n-1}[\ell]/\Gamma_{n-1}[\ell]$ via the quotient map. The image is contained in the preimage from the above commutative diagram. Now let $x\in \Gamma_{n-1}[\ell]$ be in the preimage, i.e.\ $x$ maps to a permutation $y\in S_n \subset \Gamma_{n-1}[\ell]/\Gamma_{n-1}[\ell]$ via the quotient map. Let $z \in B_n[\ell]$ be a preimage of $y$, then the difference $\rho(z)^{-1}x$ lies in $\Gamma_{n-1}[2\ell]$. Let $w\in B_n[2\ell]$ map to $\rho(z)^{-1}x$ via $\rho$. Then $z\cdot w \in B_n[\ell]$ maps to $x$ via $\rho$ as
\[ \rho(z \cdot w) = \rho(z) \cdot \rho(w) = \rho(z) \cdot \rho(z)^{-1}x = x.\qedhere\]
\end{proof}

\subsection{Reduced congruence subgroups and analogs to the main theorems}

In this final section, we want to consider the reduced integral Burau representation instead of the unreduced one and prove analogs of our main theorems. For this sake, define
\[ \Gamma'_{n-1}[\ell] = \{ A\in \Gamma'_{n-1} \mid A \equiv I_{n-1} \mod\ell\} \unlhd \Gamma'_{n-1}\]
and
\[ B'_n[\ell] := \ker( B_n \stackrel{\bar \rho}{\longrightarrow} \Gamma'_{n-1} \longrightarrow \Gamma'_{n-1}/\Gamma'_{n-1}[\ell] ).\]

The analog of \autoref{thm:mainresult} is the following theorem.

\begin{theorem}
For an integer $\ell=2^{k}m$ with $m$ odd,
   \[B_{n}/B'_{n}[\ell]\cong \begin{cases}
       \Gamma'_{n-1}/\Gamma'_{n-1}[\ell] &\text{ for $n=2,3$,}\\
        B_n/B'_n[2^k] \times\Gamma'_{n-1}/\Gamma'_{n-1}[m]  &\text{ for $n\geq 4$}.
        \end{cases}
\]
And for $n\ge4$, $B_n/B'_n[2^k]$ is trivial for $k=0$, isomorphic to $S_n$ for $k=1$, and for $k\ge2$ it is the  extension of $S_n$ by $\Gamma'_{n-1}[2]/\Gamma'_{n-1}[2^k]$ given by the preimage of $S_n\le \Gamma'_{n-1}/\Gamma'_{n-1}[2]$ in $\Gamma'_{n-1}/\Gamma'_{n-1}[2^k]$ via the quotient map. 
\end{theorem}

\begin{proof}
Since $\Gamma'_{n-1} \cong \Gamma_{n-1}$ for odd $n$, $B'_n[\ell] = B_n[\ell]$ in this case and there is nothing to prove.

For $n=2$, the reduced Burau representation is zero, so both $B_2/B'_2[\ell]$ and $ \Gamma'_{1}/\Gamma'_{1}[\ell] $ are trivial groups. Let us now assume that $n\ge 4$ and even

Recall that $\psi\colon \Gamma_{n-1} \to \Gamma'_{n-1}$ is surjective due to \autoref{lem:SESGammaodd}. It is easy to see that $\psi(\Gamma_{n-1}[\ell]) \subset \Gamma'_{n-1}[\ell]$. It follows that $\psi$ descends to a surjection $\Gamma_{n-1}/\Gamma_{n-1}[\ell]  \to\Gamma'_{n-1}/\Gamma'_{n-1}[\ell]$.  We can summarize this in the following commutative diagram.
\[\begin{tikzcd}
& \Gamma_{n-1} \arrow[rr,two heads] \arrow[dd,"\psi",two heads] && \Gamma_{n-1}/\Gamma_{n-1}[\ell] \arrow[dd,two heads] \\
B_n \arrow[ru,"\rho"]\arrow[rd,"\bar\rho"]&\\
& \Gamma'_{n-1} \arrow[rr,two heads] && \Gamma'_{n-1}/\Gamma'_{n-1}[\ell] 
\end{tikzcd}\]
This means that 
\[ B_n \stackrel{\bar \rho}{\longrightarrow} \Gamma'_{n-1} \longrightarrow \Gamma'_{n-1}/\Gamma'_{n-1}[\ell]\]
is surjective if $\ell$ is odd by \autoref{thm:mainresult} and thus $B_{n}/B'_{n}[\ell]\cong \Gamma'_{n-1}/\Gamma'_{n-1}[\ell]$ in this case.

We note that \autoref{prop:braidcong} and \autoref{cor:multbraid} can be proved completely analogously for $B'_n[\ell]$. Therefore, it remains to prove that $B_n/B'_n[2^k]$ is a non-split extension of $S_n$ by $\Gamma'_{n-1}[2]/\Gamma'_{n-1}[2^k]$. First, we note that there is the following commutative diagram.
\[\begin{tikzcd}
&  & \Gamma_{n-1}/\Gamma_{n-1}[2] \arrow[dd,two heads] \\
B_n \arrow[rru,"\rho",bend left=20]\arrow[rrd,"\bar\rho",bend right=20,swap]\arrow[r,two heads]& S_n\arrow[ru,hook]\arrow[rd,hook]\\
&  & \Gamma'_{n-1}/\Gamma'_{n-1}[2]
\end{tikzcd}\]
The image of $\rho$ is clearly the permutation matrices modulo $2$ in $\Gamma_{n-1}/\Gamma_{n-1}[2]$. It remains to prove that the composition $S_n \to \Gamma_{n-1}/\Gamma_{n-1}[2] \to \Gamma'_{n-1}/\Gamma'_{n-1}[2]$. It is easy to see that this is the standard representation of $S_n$ on $\mathbb F_2^{n-1}$, which is faithful for $n\ge 3$.

Combining this commutative diagram with the information from the previous commutative diagram, we get that the image of $B_n$ in $\Gamma'_{n-1}/\Gamma'_{n-1}[2^k]$ is precisely the preimage of $S_n$ via $\Gamma'_{n-1}/\Gamma'_{n-1}[2^k] \to \Gamma'_{n-1}/\Gamma'_{n-1}[2]$. This implies that $B_n/B'_n[2^k]$ is an extension of $S_n$ by $\Gamma'_{n-1}[2]/\Gamma'_{n-1}[2^k]$.
\end{proof}

\begin{remark}
We were not able to easily determine if the extension of $S_n$ by $\Gamma'_{n-1}[2]/\Gamma'_{n-1}[2^k]$ split or not.
\end{remark}

The analog of \autoref{thm:prob3.4} is the following theorem.

\begin{theorem}
The image of $B'_n[\ell]$ in $\GL_{n-1}(\Z)$ under the reduced integral Burau representation is completely characterized as follows.
\begin{enumerate}
    \item If $\ell$ is even or $n=2,3$, the image is \[\Gamma'_{n-1}[\ell].\]
    \item If $n\ge 4$ and $\ell$ is odd the image is the preimage of 
    \[\begin{tikzcd}S_n \arrow[r, hook]& \Gamma'_{n-1}[\ell]/\Gamma'_{n-1}[2\ell] \end{tikzcd}\]
 along the quotient map $\Gamma'_{n-1}[\ell]\longrightarrow \Gamma'_{n-1}[\ell]/\Gamma'_{n-1}[2\ell].$
\end{enumerate}
\end{theorem}

\begin{proof}
Again, for $n$ odd, there is nothing to prove as $\Gamma_{n-1} \cong \Gamma'_{n-1}$ in that case.

We shall assume that $n$ is even. It is enough to prove that $\Gamma_{n-1}[\ell] $ surjects onto $\Gamma'_{n-1}[\ell]$. To see this we consider the following map of short exact sequences.
\begin{equation*}
\begin{tikzcd}
 1 \arrow[r]&  \Gamma_{n-1}[\ell]\arrow[d]\arrow[r] &  \Gamma_{n-1}\arrow[d, two heads] \arrow[r]  &\Gamma_{n-1}/\Gamma_{n-1}[\ell]\arrow[d, two heads]\arrow[r]&  1 \\
1 \arrow[r]& \Gamma'_{n-1}[\ell]\arrow[r] &  \Gamma'_{n-1}\arrow[r]  &\Gamma'_{n-1}/\Gamma'_{n-1}[\ell]\arrow[r]&  1 
\end{tikzcd}
\end{equation*}
By \autoref{lem:SESGammaodd}, the kernel of $\Gamma_{n-1} \twoheadrightarrow \Gamma'_{n-1}$ is $\Z$ and analogously the kernel of $\Gamma_{n-1}/\Gamma_{n-1}[\ell] \twoheadrightarrow \Gamma'_{n-1}/\Gamma'_{n-1}[\ell]$ is $\Z/\ell\Z$. The map between these kernels is simply the surjective quotient map. The snake lemma implies that $\Gamma_{n-1}[\ell] \to \Gamma'_{n-1}[\ell]$ is surjective.
\end{proof}

\bibliographystyle{alpha}
\bibliography{bibliography}{}

\end{document}